\documentclass[12pt]{amsart}
\usepackage{latexsym,enumerate}
\usepackage{amssymb}
\usepackage{mathrsfs}
\usepackage[colorlinks]{hyperref}

\newtheorem{theorem}{Theorem}[section]
\newtheorem{lemma}[theorem]{Lemma}

\theoremstyle{definition}

\theoremstyle{remark}

\numberwithin{equation}{section}

\newcommand{\GL}{{\mathrm {GL}}}

\newcommand{\Aut}{{\mathrm {Aut}}}

\newcommand{\Irr}{{\mathrm {Irr}}}

\newcommand{\Soc}{{\mathrm {Soc}}}

\newcommand{\Stab}{{\mathrm {Stab}}}

\newcommand{\bF}{{\mathbf{F}}}
\newcommand{\GF}{\mbox{GF}}

\newcommand{\bC}{{\mathbf{C}}}

\newcommand{\bN}{{\mathbf{N}}}
\newcommand{\bZ}{{\mathbf{Z}}}

\newcommand{\Sy}{\textup{\textsf{S}}}

\begin{document}

\title[Odd-degree representations]
{A lower bound for the number of odd-degree\\ representations of a
finite group}

\author[N.\,N. Hung]{Nguyen Ngoc Hung}
\address{Department of Mathematics, The University of Akron, Akron,
OH 44325, USA} \email{hungnguyen@uakron.edu}

\author[T.\,M. Keller]{Thomas Michael Keller}
\address{Department of Mathematics, Texas State University, 601 University Drive, San Marcos, TX 78666, USA}
\email{keller@txstate.edu}

\author[Y. Yang]{Yong Yang}
\address{Department of Mathematics, Texas State University, 601 University Drive, San Marcos, TX 78666, USA}
\email{yang@txstate.edu}

\thanks{%We are grateful to the referee for suggesting us to look
%at an asymptotic bound for the number of odd-degree representations
%of a finite group. Theorem 1.1 and its proof arise from his/her
%suggestions and ideas.
This work was initiated when the first author
visited the Department of Mathematics at Texas State University
during Fall 2017, and he would like to thank the department for its
hospitality. The third author is partially supported by a grant from
the Simons Foundation (No 499532)}



%\date{\today}

\begin{abstract}
Let $G$ be a finite group and $P$ a Sylow $2$-subgroup of $G$. We
obtain both asymptotic and explicit bounds for the number of
odd-degree irreducible complex representations of $G$ in terms of
the size of the abelianization of $P$. To do so, we, on one hand,
make use of the recent proof of the McKay conjecture for the prime 2
by Malle and Sp\"{a}th, and, on the other hand, prove lower bounds
for the class number of the semidirect product of an odd-order group
acting on an abelian $2$-group.
\end{abstract}

\maketitle

%%%%%%%%%%%%%%%%%%%%%%%%%%%%%%%%%%%%%%%%%%%%%%%%%%%%%%%%%%%%%%%%%%%%%%

\section{Introduction}

This paper is concerned with the problem of bounding the number of
$p'$-degree irreducible (complex) representations of a finite group
in terms of its $p$-local structure, where $p$ is a prime.

The number of irreducible representations of $p'$-degree of a finite
group is one of the important invariants that is predicted to be
determined locally from the $p$-local information of the group.
Indeed, the well-known McKay conjecture
\cite{McKay,Isaacs-Malle-Navarro} asserts that
\[
|\Irr_{p'}(G)|= |\Irr_{p'}(\bN_G(P))|,
\]
where $\bN_G(P)$ is the normalizer of a Sylow $p$-subgroup $P$ in
$G$ and, as usual, $\Irr_{p'}(G)$ denotes the set of irreducible
representations of $p'$-degree of $G$. Although the conjecture is
still unsolved in general, it was confirmed for all finite groups
and the prime $p=2$, thanks to the work Malle and Sp\"{a}th
\cite{Malle-Spath}.

Note that every irreducible character of $\bN_G(P)$ of degree
coprime to $p$ must lie over a linear character of $P$. Moreover,
every irreducible character of $\bN_G(P)/P'$ has $p'$-degree.
Therefore,
\[|\Irr_{p'}(\bN_G(P))|=|\Irr_{p'}(\bN_G(P)/P')|=|\Irr(\bN_G(P)/P')|,\]
which is the same as $k(\bN_G(P)/P')$, the number of conjugacy
classes of $\bN_G(P)/P'$.

In view of \cite{Malle-Spath}, in this paper, we focus on the prime
$p=2$ and attempt to bound $|\Irr_{2'}(G)|$ in terms of $|P/P'|$ -
the size of the abelianization of $P$. Our first result is an
asymptotic bound.

\begin{theorem}\label{theorem11}
Let $G$ be a finite group and $P$ a Sylow $2$-subgroup of $G$. Then
the number of odd-degree irreducible representations of $G$ is at
least $|P/P'|^\alpha$, where $\alpha>0$ is a universal constant.
\end{theorem}

The aforementioned result of Malle and Sp\"{a}th allows us to reduce
Theorem~\ref{theorem11} to the problem on bounding the class number
of a semidirect product of an odd-order group acting on an abelian
$2$-group. We then make use of Feit-Thompson's theorem on
solvability of odd-order groups, the Hartley-Turull lemma
\cite[Lemma 2.6.2]{hartley-turull} on isomorphic coprime group
actions, and a result of the second-named author
(\cite[Theorem~B]{Keller}) on bounding class numbers of solvable
groups with trivial Frattini subgroup. Unfortunately, this approach
does not provide us an explicit value of the bounding constant
$\alpha$.

The problem of bounding the class number $k(G)$ of a finite group
$G$ in terms of $|G|$ is important in Group Theory and has a long
history, dating back to the work of Landau \cite{Landau}. Though it
has been studied extensively in the literature, it is still open
whether $k(G)$ is logarithmically bounded below by $|G|$ (in fact,
it has been believed that $k(G)\geq\log_3|G|$ for every $G$), see
\cite{Keller,Maroti16,Baumeister-Maroti-TongViet} for the latest
results and discussion.

Combining the class number conjecture with the McKay conjecture, we
see that the number of $p'$-degree irreducible characters of a
finite group $G$ should be bounded below by $\log_3|P/P'|$ (or even
$\log_2|P/P'|$ as in the case $p=2$ below), where $P$ is a Sylow
$p$-subgroup of $G$. We confirm this expectation for $p=2$. The
following result, although asymptotically weaker than
Theorem~\ref{theorem11}, offers an explicit bound (and much better
bound when $|P/P'|$ is small) for $|\Irr_{2'}(G)|$.

\begin{theorem}\label{theorem1}
Let $G$ be a finite group and $P$ a Sylow $2$-subgroup of $G$. Then
the number of odd-degree irreducible representations of $G$ is
greater than $\log_2(|P/P'|)$.
\end{theorem}

Theorem \ref{theorem1} can be considered as a generalization of a
classical result of Burnside's theorem stating that every even-order
group has a nontrivial irreducible representation of odd degree.

As explained above, Theorems~\ref{theorem11} and \ref{theorem1} are
consequences of the following result on bounding the class number.

\begin{theorem}\label{theorem2}
Let $G$ be an odd-order group acting on an abelian group $A$ of
order $2^n$. Then $k(GA)>n$. Moreover, $k(GA)\geq 2^{\alpha n}$
where $\alpha>0$ is a universal constant.
\end{theorem}

Let $f(n)$ be the best bounding function in Theorem~\ref{theorem2}.
To give the reader a sense of this function at some small values, we
have found that $f(0)=1$, $f(1)=2$, $f(2)=4$, $f(3)=8$, $f(4)=8$,
$f(5)=16$, $f(6)=16$, and $f(7)=32$. This suggests that the values
of $f$ are perhaps always $2$-powers, and it would be interesting to
confirm this phenomenon.

Proving $p$-odd versions of Theorems~\ref{theorem11}, \ref{theorem1}
and \ref{theorem2} appears to be a hard problem, one reason being
that the group $G$ in Theorem~\ref{theorem2} is not necessarily
solvable anymore. However, although several results we prove along
the way are specified to odd-order group actions only, we expect
that the ideas can be developed further to treat other groups.

To end this introduction we remark that Malle and Mar\'{o}ti
\cite{Malle-Maroti} have recently proved that $|\Irr_{p'}(G)|\geq
2\sqrt{p-1}$ for every prime $p$ and every finite group $G$ of order
divisible by $p$. Our proposed bound $|\Irr_{p'}(G)|> \log_2 |P/P'|$
is consistent with, but in many cases significantly stronger than,
Malle-Mar\'{o}ti's bound, as shown by the case $p=2$ in
Theorem~\ref{theorem1}.

%%%%%%%%%%%%%%%%%%%%%%%%%%%%%%%%%%%%%%%%%%%%%%%%%%%%%%%%%%%%%%%%%%%

\section{Some preliminary results}

In this section we collect some preliminary lemmas that will be
needed in the proof of the main results.

We start with a well-known result on bounding the order of an
odd-order permutation group.

\begin{lemma}\label{lem1}
Let $G$ be an odd-order subgroup of the symmetric group $\Sy_m$.
Then $|G| \leq (\sqrt 3)^{m-1}$.
\end{lemma}

\begin{proof} See \cite{Dixon} for a group-theoretic proof and
\cite{Alspach} for a combinatorial proof. Note that the equality
holds if and only if $m$ is a power of $3$ and $G$ is a Sylow
$3$-subgroup of $\Sy_m$.
\end{proof}

The next two lemmas are important properties of odd-order group
actions.

\begin{lemma}\label{lemmaDolfi} Let $H$ be a (solvable) group of odd order, and let $V$ be a
finite, faithful and irreducible $H$-module. Then
\begin{itemize}
\item[(i)] $|V|$ cannot be the square of a Mersenne prime;

\item[(ii)] if $|V|=25$, then $|H|=3$;

\item[(iii)] if $|V|=81$, then $|H|=5$.
\end{itemize}
\end{lemma}

\begin{proof}
This follows from \cite[Lemma 2.4]{SD}.
\end{proof}

\begin{lemma}\label{lemma-manz-wolf}
Let $G$ be an odd-order group and $V\neq 0$ be a faithful,
completely reducible and finite $G$-module. If $char(V)$ is odd,
then $|G|\leq |V|^{1.5}/(24)^{1/3}$.
\end{lemma}

\begin{proof}
This follows from \cite[Theorem 3.5]{manz-wolf}.
\end{proof}

As usual we use $\bF(G)$ to denote the Fitting subgroup of $G$,
which is the minimal nilpotent normal subgroup of $G$, and $\Soc(G)$
to denote the subgroup generated by all minimal normal subgroups of
$G$, which is usually referred to as the socle of $G$.

The following provides detailed structure of an odd-order group that
acts faithfully, irreducibly, and quasi-primitively on a finite
vector space.

\begin{lemma}  \label{lem2}
Suppose that $V$ is a faithful, irreducible, and quasi-primitive
module for an odd-order group $G$ over the Galois field $\GF(q)$,
where $q$ is a prime power. Then $G$ has normal subgroups $Z
\subseteq T \subseteq \bF(G) \subseteq A\subseteq G$ such that:

\begin{itemize}
\item[(i)] $F:=\bF(G)$ is a central product $F=ET$ with
$E,T \triangleleft G$ and $T$ is cyclic, $Z=E \cap
T=\bZ(E)=\Soc(\bZ(F))$ and $T=\bZ(F)$, and Sylow subgroups of $E$
are of prime order or extra-special of prime exponent;

\item[(ii)] $A=\bC_G(Z)$ and, in particular, $G/A$ is isomorphic to a subgroup of $\Aut(Z)$;

\item[(iii)] $T=\bC_G(E)$, $\bF(G)=\bC_A(E/Z)$, and $E/Z$ is a faithful
completely reducible symplectic $A/\bF(G)$-module;

%\item Each Sylow subgroup of $E/Z$ is a completely reducible $A/\bF(G)$-module;

\item[(iv)] $|F:T|=e^2$ where $e$ is an integer dividing $\dim(V)$;

\item[(v)] if $W$ is an irreducible $T$-submodule of $V$,
then $T$ acts fixed point freely on $W$ and thus $|T|$ divides
$|W|-1$;

\item[(vi)] $|V|=|W|^{be}$ for some integer $b$ and $|G:A| \mid \dim(W)$.
%\item Set $|W|=q^m$, then $mb=a|G:A|$;
\end{itemize}
\end{lemma}

\begin{proof}
From the hypothesis and by the Clifford theorem, every normal
abelian subgroup of $G$ is cyclic. The lemma then follows from
\cite[Corollary 1.10]{manz-wolf} and \cite[Corollary
2.6]{manz-wolf}.
\end{proof}

\begin{lemma}  \label{lem3}
Suppose that $G$ is an odd-order (solvable) irreducible subgroup of
$\GL(n,q)$ and the action of $G$ on the natural vector space is
quasi-primitive. Let $e$ be defined as in Lemma~\ref{lem2}. Then $e
\neq 3, 7$.
\end{lemma}

\begin{proof}
Assume by contradiction that $e=3$ or $e=7$. We will keep the
notation as in Lemma~\ref{lem2}. We have $\bF(G)/T \cong E/Z$ is a
direct sum of the chief factors of $G$ and each factor, as an
irreducible symplectic $G$-module, has even dimension (see
\cite[Corollary 1.10]{manz-wolf}). It follows that
$G/\bC_G(\bF(G)/T)$ is a solvable irreducible subgroup of $\GL(2,3)$
or $\GL(2,7)$. Using Lemma~\ref{lemmaDolfi}, we deduce that
$|G/\bC_G(\bF(G)/T)|$ is even, a contradiction.
\end{proof}

We need to use \emph{semi-linear groups} and, for the reader's
convenience, we recall the definition here.

Let $V$ be the Galois field $GF(q^m)$ for a prime power $q$. For
fixed $a\in V^{\#}:=V\setminus \{0\}$, $w\in V$ and $\sigma \in
\mathrm{Gal}(GF(q^m)/GF(q))$, one can define the mapping
\begin{center}
$T: V\rightarrow V$ by $T(x)=ax^{\sigma}+w$,
\end{center}
which turns out to be a permutation on $V$. Note that $T$ is trivial
if and only if $a=1$, $\sigma =1$ and $w=0$. The affine semi-linear
group of $V$ is defined to be the group of all those permutations,
and the semi-linear group is a subgroup of the affine semi-linear
group:
\[\Gamma(V)=\Gamma(q^m)=\{x\rightarrow ax^{\sigma}\ |\ a\in
GF(q^m)^{\#}, \sigma \in \mathrm{Gal}(GF(q^m)/GF(q))\}.\]

To end this section, we record a result of S. Seager \cite[Theorem
1]{seager}. We restate it here in the form suitable for our
purposes.

\begin{theorem} \label{thmseager}
Let $p$ be a prime and $G$ a finite solvable group acting faithfully
and irreducibly on the finite $\GF(p)$G-module $V$.  Write $n(G,V)$
for the number of orbits of $G$ on $V$. Then by a classical result
(see \cite[Theorem 6]{seager}) we know that $G$ is isomorphic to a
subgroup of $T\wr S$, where $T$ is a solvable irreducible subgroup
of $\GL(m, p)$ and $S$ is a solvable transitive subgroup of $\Sy_n$
for some integers $m$, $n$; moreover,  $\dim V=mn$. Then the
following hold.
\begin{itemize}
\item[(i)] If $T$ (as a linear group) is isomorphic to a subgroup of
$\Gamma(p^m)$, then
\[n \leq\   0.157 \log_3((n(G,V) + 1.43)/24^{1/3}).\]

\item[(ii)] If $T$ (as a linear group) is not isomorphic to a subgroup of
$\Gamma(p^m)$, then
\[|V|=p^{mn} < ((n(G,V) + 1.43)/24^{1/3})^c\]
for $c = 36.435663$.
\end{itemize}
\end{theorem}

\begin{proof}
This follows from \cite[Section 6]{seager} by changing the language
of primitive permutation groups to the language of linear groups.
\end{proof}

\section{The asymptotic bound}

In this section we prove Theorem~\ref{theorem11}. We start with the
asymptotic part of Theorem~\ref{theorem2}. We thank the referee for
pointing out this to us.

\begin{theorem}\label{theorem-111}
Let $G$ be an odd-order group acting on an abelian $2$-group $A$.
Then $k(GA)\geq |A|^\alpha$ for some universal constant $\alpha>0$.
\end{theorem}

\begin{proof}
 First we observe that $G$ is solvable by Feit-Thompson's theorem \cite{FT63}.

 Suppose that the action of $G$ on $A$ is not faithful. Let $N:=\bC_G(A)$.
Then $1<N\unlhd GA$ and $G/N$ acts faithfully on $A$. We then have
$k(GA)\geq k((GA)/N)=k((G/N)A)$. So we may assume that $G$ is
faithful on $A$. It follows that $A$ is the Fitting subgroup of
$GA$.

A well-known formula for $k(GA)$ states that
\[ k(GA)\ =\ \sum\limits_{i=1}^{n(G,A)}k(\bC_G(a_i)),\]
where $n(G,A)$ is the number of orbits of $G$ on $A$, and the
$a_i\in A$ ($1\leq i \leq n(G,A)$) are representatives of the orbits
of $G$ on $A$, see \cite[Proposition 3.1]{Smith07} for instance.

Now by the Hartley-Turull Lemma \cite[Lemma 2.6.2]{hartley-turull},
there exists an elementary abelian group $B$ such that the actions
of $G$ on $A$ and $B$ are permutation isomorphic. Now $|A|=|B|$ and
$n(G,A)=n(G, B)$, and if $\phi: A\to B$ is the corresponding
permutation isomorphism, then $\bC_G(a)=\bC_G(\phi(a))$ for every
$a\in A$. Hence by the $k(GA)$-formula above it follows that
$k(GA)=k(GB)$, and so by replacing  $A$ by $B$ if necessary we may
assume that $A$ is elementary abelian. Therefore the Frattini
subgroup of $GA$ is trivial.

We have shown that $GA$ is a solvable group with trivial Frattini
subgroup. Now using \cite{Keller}, we deduce that there exists a
universal constance $\alpha>0$ such that
\[k(GA)\geq |GA|^{\alpha}\geq |A|^{\alpha},\] as
desired.
\end{proof}

\begin{proof}[Proof of Theorem \ref{theorem11}] As explained in the introduction, we have
\[
|\Irr_{2'}(G)|=k(\bN_G(P)/P'),
\] by Malle-Sp\"{a}th's theorem. By the Schur-Zassenhaus theorem, $\bN_G(P)/P'$
is a split extension of the abelian $2$-group $P/P'$ and the
odd-order group $\bN_G(P)/P$. Therefore Theorem~\ref{theorem11}
follows by Theorem~\ref{theorem-111}.
\end{proof}

%%%%%%%%%%%%%%%%%%%%%%%%%%%%%%%%%%%%%%%%%%%%%%%%%%%%%%%%%%%%%%%

\section{Toward the explicit bound}

We have to work much harder to obtain an (good) explicit bound for
$|\Irr_{2'}(G)|$. The main result of this section is
Theorem~\ref{thm3}, which will be used in a key step in the proof of
Theorem~\ref{theorem2} in the final section.

Recall that an orbit $x^G$ of an action of a group $G$ on a set $X$
is said to be \emph{regular} if $|x^G|=|G|$, or equivalently,
$\bC_G(x)$ is trivial.

It was proved in \cite{Yang1} that if $V$ is a finite, faithful,
irreducible, and quasi-primitive $G$-module for an odd-order group
$G$ with $\bF(G)$ nonabelian (i.e. the action of $G$ on $V$ is not
semi-linear), then $G$ has at least $212$ regular orbits on $V$.
(The bound $212$ is indeed sharp.) In the following important
result, we improve this to a bound in terms of $\dim(V)$.

\begin{theorem} \label{thm2}
Let $G$ be a (solvable) group of odd order. Let $V$ be a finite,
faithful, irreducible, and quasi-primitive $G$-module. Assume that
the action of $G$ on $V$ is not semi-linear. Then $G$ has at least
$\max (212, k+1)$ regular orbits on $V$, where $k= \dim(V)$.
\end{theorem}

\begin{proof}
By \cite[Theorem 1.1]{Yang1}, we may assume that $\dim(V)=k \geq
212$. We will follow some ideas in the proof of \cite[Theorem
2.2]{SD}. Assume  that $G$ is a counterexample of minimal order.
Then $G$ has at most $k$ regular orbits on $V$.
%By Lemma ~\ref{lem5}, either $|V|> 422^{\frac 9 4}$ or $(n,q)=(5,11)$ or $(n,q)=(9,4)$.

Let $G^\#$ be the set of nonidentity elements of $G$. Any $v \in V$
not contained in $\bigcup_{g\in G^\#}\bC_V(g)$ must necessarily lie
in a regular $G$-orbit, each of which has size exactly $|G|$. Since
$G$ has at most $k$ regular orbits, we deduce that
\[ |V| \leq k \cdot |G|+\left | \bigcup_{g\in G^\#}\bC_V(g) \right |. \]

By \cite[Lemma 1.3]{AE2}, we have
\[
\dim(\bC_V(g))\leq \frac{4}{9}\dim(V) \] whenever $g \neq 1$, and it
follows that \[|\bC_V(g)| \leq |V|^{\frac 4 9}\] for all $g \neq 1$.

Let $\mathcal{P}$ be the collection of all subgroups of $G$ of prime
order. Then, as $|G|$ is odd, we have $|\mathcal{P}| \leq
(|G|-1)/2$. It follows that
\[ \left |\bigcup_{g\in G^\#}\bC_V(g) \right |= \left |\bigcup_{S\in
\mathcal{P}}\bC_V(S) \right | \leq \frac {|G|-1}{2} \cdot |V|^{\frac
4 9}. \]
%Thus \[|V| \leq 64 \cdot |G|+ \frac {|G|-1}{2} |V|^{4/9} \]
%Define
%\begin{equation}
%f(|G|,|V|)= |V|-64 \cdot |G| - \frac{1}{2} (|G|-1)|V|^{4/9}
%\end{equation}
We then get \[|V| \leq k \cdot |G|+\frac {|G|-1}{2} \cdot |V|^{\frac
4 9},\] and thus \[|G| \geq \frac{2|V|+|V|^{\frac 4
9}}{2k+|V|^{\frac 4 9}}
> \frac{2|V|}{2k+|V|^{\frac 4 9}}.\]

We now adopt the notation in Lemma~\ref{lem2}.

From Lemma~\ref{lem2} we know that $T$ acts fixed point freely on
$W$ and $|T| \mid (|W|-1)$, thus $|T| < |W|$. Moreover $|T|$ is odd
and $T \neq 1$, we also have $|T| \geq 3$ and therefore $|W| \geq
4$. Lemma~\ref{lem2}(6) then implies that $|V|\geq 4^e$. Recall that
$\dim V=k$, so $|V|\geq 2^k$. Therefore we have
\[
|V|\geq \max \{ 4^e,2^k\}.
\]
Note that $e$ is a positive odd integer bigger than 1. In fact
$e\geq 5$ by Lemma~\ref{lem3}. This and the fact that $e \mid k$
imply that $\max \{ 4^e,2^k\}^{4/9}\geq 2k$. So we have
\[
|V|^{\frac4 9}>2k.
\]
Therefore \[|G|
> |V|^{\frac 5 9}.\] Since $|V|=|W|^{be}$, we now have \[|G| >
|W|^{\frac{5be} {9}}.\]

Recall from Lemma~\ref{lem2}(iii) that $F/T$ is a faithful
completely reducible symplectic $A/F$-module. Using \cite[Theorem
A]{AE2}, we deduce that $F/T$ contains at least two regular
$A/F$-orbits. Note that $|F/T|=e^2$. Hence
\[
|A/F|\leq |F/T|/2=e^2/2.
\]

%$G/A$ is isomorphic to an odd subgroup of $\Aut(Z)$ and $Z$ is a
%direct product of cyclic groups of order $p_i$, where $p_i$ are odd
%primes. Put all these together, we will finally get that $|G| \leq
%e^4 \cdot |T|^2/4$.
From Lemma~\ref{lem2}(vi) that $|G/A| \mid \dim W$, we have $|G/A|
\leq \log_2 |W|$. Recall that $|T|<|W|$. We now have
\begin{align*}|G|=&|G/A||A/F||F/T||T|\\
 <& \log_2 |W| \cdot e^4/2 \cdot |W|. \end{align*}
Therefore
\[|W|^{5e/9} < \log_2 |W| \cdot e^4 \cdot |W|/2.\]

Recall from Lemma~\ref{lem3} that $e \neq 3$ and $e \neq 7$. Let
$p_1, \dots, p_l$ be all the different prime divisors of $e$. Then
$p_1 \dots p_l \mid |Z| \mid |T| \mid (|W|-1)$. Take these into
consideration and solve the above inequality, we come up with only
two cases: $e=5$, $|W| \leq 70$ and $e=9$, $|W| \leq 10$.

\begin{enumerate}

\item Assume $e=5$ and $|W| \leq 70$. Thus $|F/T|=e^2=25$ and
$|A/F|=3$ by Lemma~\ref{lemmaDolfi}(ii). Moreover, $|G/A|$ is odd
and $G/A \lesssim \Aut(Z)$. Since $|W|$ is a prime power and $5 \mid
|T| \mid (|W|-1)$, we may assume $|T|=5$ or $15$.

\begin{enumerate}
\item Assume $|T|=5$. Thus $Z \cong Z_5$, $G/A \lesssim \Aut(Z) \cong Z_4$
and hence $|G/A|=1$ as $|G/A|$ is odd. Therefore,
\[|G|=|G/A||A/F||F/T||T| = 3 \cdot 5^3.\]
On the other hand, since $|T| \mid (|W|-1)$, we have $|W| \geq 11$.
Also, since $\dim(V) \geq 5 \cdot \dim(W) $, $|V| \geq 11^5$. But
then we have $|G| \leq |V|^{5/9}$, a contradiction.

\item Assume $|T|=15=5 \cdot 3$. Thus $Z \lesssim Z_3 \times Z_5$,
$G/A \lesssim \Aut(Z) \cong Z_2 \times Z_4$ and again $|G/A|=1$. It
follows that
\[|G|=|G/A||A/F||F/T||T| = 3 \cdot 5^3 \cdot 3.\]
Since $|T| \mid (|W|-1)$, we have $|W| \geq 16$. Moreover, since
$\dim(V) \geq 5 \cdot \dim(W) $, we get $|V| \geq 16^5$. But then we
have $|G| \leq |V|^{5/9}$, a contradiction again.

%\item Assume $|T|=25=5 \cdot 5$, then $Z=Z_5$. $G/A \subseteq
%\Aut(Z)=Z_4$ and $|G/A|=1$. \[|G|=
%|G/A||A/F||F/T||T| \leq 3 \cdot 5^3 \cdot 5 \] Since $|T| \mid
%|W|-1$, thus $|W| \geq 101$ and $|V| \geq 101^5$. $|G| \leq
%101^{2.5} \leq |V|^{1/2}$, contradiction.

%\item Assume $|T|=35=5 \cdot 7$, then $Z=Z_5 \times Z_7$. $G/A
%\subseteq \Aut(Z)=Z_4 \times Z_6$ and $|G/A| \leq 3$.
%\[|G|=|G/A||A/F||F/T||T| \leq 9 \cdot 5^3 \cdot 7\]
%Since $|T| \mid |W|-1$, $|W| \geq 71$. $\dim(V)=5 \cdot \dim(W)
%$, so $|V| \geq 71^5$. $|G| \leq 71^{2.5} \leq |V|^{1/2}$,
%contradiction.

%\item Assume $|T|=65=5 \cdot 13$, then $Z=Z_5 \times Z_{13}$. $G/A
%\subseteq \Aut(Z)=Z_4 \times Z_{12}$ and $|G/A| \leq 3$.  %\[|G|=|G/A||A/F||F/T||T|\leq 9 \cdot 5^3 \cdot 13 \]
%Since $|T| \mid |W|-1$, $|W| \geq 131$. $\dim(V)=5 \cdot
%\dim(W) $, so $|V| \geq 131^5$. $|G| \leq 131^{2.5} \leq |V|^{1/2}$,
%contradiction.

%\item Assume $|T|=75=5 \cdot 15$, then $Z=Z_3 \cdot Z_5$. $G/A
%\subseteq \Aut(Z)=Z_2 \times Z_4$ and $|G/A|=1$.
%\[|G|=|G/A||A/F||F/T||T| \leq 3 \cdot 5^3 \cdot 15\]
%Since $|T| \mid |W|-1$, $|W| \geq 151$. $\dim(V) \geq 5 \cdot
%\dim(W) $, so $|V| \geq 151^5$. $|G| \leq 151^{2.5} \leq |V|^{1/2}$,
%contradiction.

\end{enumerate}

\item Assume $e=9$ and $|W| \leq 10$. Then $|F/T|=e^2=81$. Now the action of $A/F$ on $F/T$ is
irreducible (since otherwise $|A/F|$ is even) and then $|A/F| \leq
5$ by Lemma~\ref{lemmaDolfi}. Since $|W|$ is a prime power and $3
\mid |T| \mid (|W|-1)$, we have $|W|=4$ or $7$, $|T|=3$ and $T = Z
\cong Z_3$. It follows that $G/A \lesssim \Aut(Z) \cong Z_2$ and
hence $G/A=1$.

Recall that $\dim(V) \geq 212$, we know that $|W|=7$, or $|W|=4$ and
$b \geq 2$. We have $|G|=|G/A||A/F||F/T||T| \leq 5 \cdot 3^5 \leq
|V|^{5/9}$, a contradiction.
\end{enumerate}
\end{proof}

%The following could be viewed as a strengthened version of ~\cite[]{manz-wolf}.

If $G$ is a permutation group on a set $\Omega$, there will be an
induced action of $G$ on the power set $\mathscr{P}(\Omega)$. The
$G$-orbit of a subset $\Delta \subseteq \Omega$ is \emph{regular} if
$\Stab_G(\Delta)$ is trivial. If we have $|\Delta|\neq |\Omega|/2$
in addition, the orbit is called \emph{strongly regular}.

\begin{lemma}  \label{lemoddperm}
Let $G$ be an odd-order transitive primitive permutation group on
$\Omega$ with $|\Omega|=n$. Then $G$ has at least $\lceil \frac n
{25} \rceil$ strongly regular orbits on the power set
$\mathscr{P}(\Omega)$ of $\Omega$.
\end{lemma}

\begin{proof}
We recall the structure of a solvable primitive permutation group.
Let $V$ be a minimal normal subgroup of $G$ and let $S$ denote a
point stabilizer. We know that $G=VS$, $S \cap V=1$, and
$\bC_G(V)=V$. Also $n=|\Omega|=|V|$ is a prime power, and we set
$n=p^t$ where $p$ is a prime. For $g\in G$, we denote by $n(g)$ the
number of cycles of $g$ on $\Omega$, by $o(g)$ the smallest prime
divisor of the order of $g$, and by $s(g)$ the number of fixed
points of $g$. We set $o$ to be the smallest order of all the
nontrivial elements in $G$.

For a subset $X \subseteq G$, we consider the following set
\[S(X)=\{(g,\Gamma) \ |\ g \in X, \Gamma \subseteq \Omega, g \in
\Stab_G(\Gamma)\}.\] It is easy to see that in order to estimate
$S(G\backslash \{1\})$, we only need to count elements of prime
order. Also, assume $g_1$ is of prime order and let $g_2 \in \langle
g_1 \rangle$, then the subsets of $\Omega$ stabilized by $g_1$ and
$g_2$ are exactly the same. Note that $g \in G$ stabilizes exactly
$2^{n(g)}$ subsets of $\Omega$.

Now we have
\[|S(G \backslash \{1\})| \leq \frac {|G|} {o-1} \cdot 2^{\lfloor (p+o-1)n/(o \cdot p) \rfloor}.\]
Therefore, it is sufficient to show that
\[\frac {|G|} {o-1} \cdot 2^{\lfloor (p+o-1)n/(o \cdot p) \rfloor} <
2^{n}-\left\lfloor {\frac n {25}} \right\rfloor \cdot |G|, \] which
is equivalent to
\[
|G| \left(\frac{2^{\lfloor (p+o-1)n/(o \cdot
p)}\rfloor}{o-1}+\left\lfloor {\frac n {25}}
\right\rfloor\right)<2^n.
\]
Since $n$ is odd and $|G|$ is odd, we know that \[|G|=|V||S| \leq
|V|\cdot |V|^{1.5}/(24)^{1/3}=n^{2.5}/(24)^{1/3}\] by
Lemma~\ref{lemma-manz-wolf}. Moreover, $(p+o-1)/(o\cdot p)\leq 5/9$
as both $o$ and $p$ are at least $3$. Thus, it is sufficient to
prove

%Let $|G|=(3^{1/2})^l$, and thus $k=\lceil l \rceil$ (the ceiling function).

%In almost all cases, we can show that for any subset $\Omega_1 \subseteq \Omega$ with $n-k$ elements, there exists a subset $\Delta \subseteq \Omega_1$ such that $G$ has a regular orbit on $\Delta$ (as $G$ acts on the power set of $\Omega$), and thus the result holds. The spirit of the proof is, look at any subsets of $\Omega$ of size $n-k$, call it $\Omega_1$, by throwing away enough bad subsets out of the power set of $\Omega_1$, there is still enough room left so there exists a subset $\Delta \subseteq \Omega_1$ such that $G$ has a regular orbit on it. In the exceptional cases, we have to do case study.

%Let $r$ to be the smallest prime divisor of $|G|$.

%In order to show that for any subset, say $\Omega_1 \subseteq \Omega$
%with $n-k$ elements, there exists a subset $\Delta \subseteq \Omega_1$
%such that $G$ has a regular orbit on $\Delta$, it suffices to show that $|S(G^{\#})| <2^{n-k}.$

\[\frac 1 {24^{1/3}} \cdot n^{2.5} \left(\frac {2^{\lfloor\frac{5n}{9} \rfloor}} 2+ \left\lfloor {\frac n {25}}
\right\rfloor\right)  < 2^{n}.\] It is not hard to see that this
inequality is satisfied when $n \geq 25$. When $n \leq 24$, we have
$\lfloor {\frac n {25}} \rfloor =0$, and the result follows by
Gluck's theorem that $G$ has at least one strongly regular orbit on
$\mathscr{P}(\Omega)$, see \cite[Theorem 5.6]{manz-wolf}.
\end{proof}

The following is the main result of this section, and it is a
crucial step in the proof of Theorem~\ref{theorem2}.

\begin{theorem} \label{thm3}
Let $G$ be a group of odd order. Let $V$ be a finite, faithful, and
irreducible $G$-module. Assume $G$ is induced from a non-semi-linear
group acting on $W\leq V$, then $G$ has at least $\max (n,212)$
regular orbits on $V$, where $n=\dim(V)$.
\end{theorem}

\begin{proof}
%Assume false and consider a counterexample with $|G|+\dim(V)$ as small as possible. Suppose that $V$ is not irreducible then $V=V_1+V_2$ where $V_1$, $V_2$ are $G$-submodules. Let $C_i=\bC_G(V_i)$ and let $K_i/C_i=\bF_2(G/C_i)$. We can find $x_i,y_i \in V_i$ such that $x_i, y_i$ are in different $G/C_i$ orbits, $\bC_G(x_i) \subseteq K_i$ and $\bC_G(y_i) \subseteq K_i$. Let $x=x_1+x_2$, $y=y_1+y_2$ and we have $\bC_G(x),\bC_G(y) \subseteq K_1 \cap K_2=\bF_2(G)$. Clearly $x$ and $y$ are not $G$-conjugate and this is a contradiction.
We first assume that $V$ is a quasi-primitive $G$-module.
Thus $G$ is not semi-linear. By Theorem~\ref{thm2}, $G$ has at least $\max (n, 212)$ regular orbits on $V$.% or $\bF(G)$ is abelian. If $\bF(G)$ is abelian, then $G \leq \Gamma(V)$ by \cite[Corollary 2.3(a)]{manz-wolf}. %Also $G$ has at least two orbits on $V$ as long as $V \neq 0$. This tells us that there exist two $G$-orbits with representatives $v_a$, $v_b \in V$ such that $\bC_G(v_a), \bC_G(v_b) \subseteq \bF_2(G)$. This is a contradiction.

Thus the action of $G$ on $V$ is not quasi-primitive and there
exists a normal subgroup $N$ of $G$ such that $V_N=V_1 \oplus \dots
\oplus V_m$ for $m>1$ homogeneous components $V_i$ of $V_N$. If $N$
is maximal with this property, then $S=G/N$ primitively permutes the
$V_i$'s. Also $V=V_1^G$, induced from $\bN_G(V_1)$. Let
$H=\bN_G(V_1)/\bC_G(V_1)$, then $H$ acts faithfully and irreducibly
on $V_1$ and $G$ is isomorphic to a subgroup of $H \wr S$.

If $x_1 \in V_1$ and $C_1$ is the $H$-conjugacy class of $x_1$, then
the only $G$-conjugates of $x_1$ in $V_1$ are the elements of $C_1$.
The set of $G$-conjugates of $x_1$ is $C_1 \cup \dots \cup C_m$
where $C_i \subset V_i$ is a $G$-conjugate of $C_1$. Choose $y_1 \in
V_1$ in an $H$-conjugate class different than that of $x_1$. Also
choose $x_i$ and $y_i$ in $V_i$ conjugate to $x_1$ and
$y_1$(respectively) for all i. Then no $x_i$ is ever $G$-conjugate
to a $y_j$. In particular, if $g \in G$ centralizes
$v=x_1+\dots+x_j+y_{j+1}+\dots+y_m$, then $g$ and $Ng$ must
stabilize $\{V_1,\dots,V_j\}$ and $\{V_{j+1},\dots,V_m\}$.

Set $\dim(V_1)=l$, and we know that $n=m \cdot l$. By induction, $H$
has at least $a=\max (l, 212)$ regular orbits on $V_1$. Say $u_1,
v_1, w_1$ are three elements in distinct such orbits. Let $u_i \in
V_i$ be $G$-conjugates of $u_1$ $(1 \leq i \leq m)$, let $v_i \in
V_i$ be $G$-conjugates of $v_1$ $(1 \leq i \leq m)$, and let $w_i
\in V_i$ be $G$-conjugates of $w_1$ $(1 \leq i \leq m)$. Suppose
that $\Omega$ can be written as a disjoint union of $A_1, A_2$ $(A_i
\neq \varnothing$, $i=1,2)$ such that $\Stab_S(A_1) \cap
\Stab_S(A_2)=1$ and $|A_1|<|A_2|$. Define $x=x_1+x_2+ \dots + x_m$
such that $x_i=u_i$ if $i \in A_1$, and $x_i=v_i$ if $i \in A_2$.
Then $\bC_G(x)=\bC_N(x)=1$. Define $y=y_1+y_2+\dots + y_m$ such that
$y_i=v_i$ if $i \in A_1$, and $y_i=u_i$ if $i \in A_2$. We see that
$x$ is not $G$-conjugate to $y$ since $|A_1| \neq |A_2|$. Define
$z=z_1+z_2+\dots + z_m$ such that $z_i=u_i$ if $i \in A_1$, and
$z_i=w_i$ if $i \in A_2$. Then $\bC_G(z)=\bC_N(z)=1$. It is easy to
see that $x$ and $z$ are not $G$-conjugate. Thus using the pair
$A_1$ and $A_2$ we may construct $a(a-1)$ regular orbits of $G$ on
$V$.

We now suppose that $\Omega$ can be written as a disjoint union of
$B_1, B_2$ $(B_i \neq \varnothing$, $i=1,2)$ such that $\Stab_S(B_1)
\cap \Stab_S(B_2)=1$ and $|B_1|<|B_2|$. Then, in a similar way, we
may construct $a(a-1)$ regular orbits of $G$ on $V$. We observe that
as long as $A_1$ and $B_1$ are in different set-orbits of $G/N$ on
the power set of $\Omega$, the regular orbits constructed using
$A_1$, $A_2$ and the ones constructed using $B_1$, $B_2$ are in
different $G$-orbits. By Lemma~\ref{lemoddperm}, there are at least
$\lceil\frac m {25}\rceil /2$ such pairs. Thus $G$ has at least
\[\frac{a(a-1) \cdot \lceil\frac m {25}\rceil}{2} \] regular orbits on $V$.
Clearly \[\frac{a(a-1) \cdot \lceil\frac m {25}\rceil}{2} \geq \max
(n,212),\] as desired.
\end{proof}
%By Theorem ~\ref{thm2}, $a \geq \max (212, k+1)$ and the results follows.

%Assume the permutation group on top is not primitive, the result follows by induction.

%%%%%%%%%%%%%%%%%%%%%%%%%%%%%%%%%%%%%%%%%%%%%%%%%%%%%%%%%%%%%%%%%%%%%

\section{The explicit bound}

We are now ready to prove the explicit bound in
Theorem~\ref{theorem2}, which is restated below.

\begin{theorem}\label{theorem3}
Let $G$ be an odd-order group acting on an abelian group $V$ of
order $2^n$. Then $k(GV)>n$.
\end{theorem}

\begin{proof}Let $(G,V$ be a counterexample such that $GV$ has minimal order.
In particular, we have $k(GV)\leq n$. We proceed in a number of
steps.\\

{\bf Step 1.} $G$ acts faithfully on $V$.\\

As in the proof of Theorem~\ref{theorem-111}, if $G$ is not faithful
on $V$ then $1<\bC_G(V)\unlhd GV$ and $G/\bC_G(V)$ acts (faithfully)
on $V$; moreover, by the choice of $G$ and $V$, we have
$k((GV)/N)=k((G/N)V
>n$, and since clearly $k(GV)\geq k((GV)/N)$, we obtain $k(GV)>n$,
contradicting
$G$ and $V$ being a counterexample. So $G$ is indeed faithful on $V$.\\

{\bf Step 2.} $V$ is elementary abelian.\\

Using the $k(GV)$-formula  $k(GV)\ =\
\sum\limits_{i=1}^{n(G,V)}k(\bC_G(v_i))$ (where $n(G,V)$ is the
number of orbits of $G$ on $V$ and the $v_i\in V$ ($1\leq i \leq
n(G,V)$) are representatives of the orbits of $G$ on $V$) and the
Hartley-Turull Lemma \cite[Lemma 2.6.2]{hartley-turull}, and arguing
similarly as in the proof of Theorem~\ref{theorem-111},
we may assume that $V$ is elementary abelian.\\

{\bf Step 3.} $V$ is an irreducible $G$-module.\\

By Steps 1 and 2 we know that $V$ can be viewed as a faithful
$G$-module over the field with two elements, and since $G$ acts
coprimely on $V$, by Maschke's theorem we know that $V$ is
completely reducible as $G$-module.

Now suppose that $V=X\oplus Y$ for two nontrivial $G$-modules $X$
and $Y$. Write $|X|=2^s$ and $|Y|=2^t$. Then $s+t=n$, $s\geq 1$ and
$t\geq 1$. By interchanging $X$ and $Y$ if necessary we may assume
that $s\leq t$. Also note that since $X$ is nontrivial, we have
$n(G,X)\geq 2$. We now apply \cite[Lemma 4.3]{kellerkgv} to deduce
that
\begin{align*}k(GV)\ &=\ \sum\limits_{i=1}^{n(G,X)}k(C_G(x_i)Y)\ >\
\sum\limits_{i=1}^{n(G,X)}t\ \\
&=\ n(G,X)\ t\ \geq 2t\ \geq\ s+t\ =\ n,\end{align*} where the
$x_i\in X$ ($1\leq i \leq n(G,X)$) are representatives of the orbits
of $G$ on $X$. This contradicts $(G,V)$ being a counterexample, and
so we now have that $V$ is irreducible as a $G$-module.\\

{\bf Step 4.} $V$ is induced from a semi-linear group action. That
is, there exists a $H\leq G$ and an irreducible $H$-module
$W$ such that $W<V$, $H/\bC_H(W)\leq\Gamma(W)$, and $V=W^G$.\\

Assume not and that $H$ is not semi-linear, then the result follows by Theorem ~\ref{thm3}.\\

{\bf Step 5.} $n\leq 15$.\\

Write $|W|=2^k$, so that $H/\bC_H(W)\leq\Gamma(2^k)$ and hence
$|H/\bC_H(W)|\leq (2^k-1)k$. Also let $N:=\cap_{g\in G}H^g\unlhd G$
be the core of $H$ in $G$, then \[V_N=W_1\oplus\dots\oplus W_m,\]
where $m=n/k$ and the $W_i$s are homogeneous $N$-modules which are
faithfully permuted by $G/N$, and we may assume that $W=W_1$. Hence
$G/N$ is isomorphic to an odd order subgroup of $S_m$, and thus by
Lemma~\ref{lem1} we have
\[|G/N|\leq (\sqrt{3})^{m-1}.\]
Moreover, clearly $|N|\leq (2^k-1)^mk^m\leq 2^nk^m$, and $N$ has an
abelian subgroup $K$ which is normal in $G$ and whose index in $N$
is at most $k^m$. Therefore $KV$ is a metabelian normal subgroup of
$G$, and by \cite[Theorem 1]{bertram} we know that
\[k(KV)>|KV|^{\frac{1}{3}}.\]
We now get \begin{align*} \ n\ &\geq\ k(GV)\geq
\frac{k(KV)}{|GV:KV|} \ \geq\ \frac{|KV|^{\frac{1}{3}}}{|G:K|}\\
&\geq\ \frac{|V|^{\frac{1}{3}}}{(\sqrt{3})^{m-1}k^m}\ \geq\
\frac{2^{\frac{n}{3}}}{(\sqrt{3}\ k)^{\frac{n}{k}}}\ >\
\frac{2^{\frac{n}{3}}}{(\sqrt{3}\ n)^{\frac{n}{k}}}. \end{align*}

We next use Seager's result Theorem \ref{thmseager}. Since $V$ is
induced from a semi-linear group, we are in Part (a) of that result.
Thus we have
\begin{eqnarray*}
\frac{n}{k}\ &\leq\ & 0.157 \log_3((n(G,V) + 1.43)/24^{1/3})\\
                  &\ <\ &0.157 \log_3((k(GV) + 1.43)/24^{1/3})\\
                  &\ \leq\ &0.157 \log_3((n + 1.43)/24^{1/3}). \\
\end{eqnarray*}

 Two inequalities in the last two paragraphs yields
\[n\ >\
\frac{2^{\frac{n}{3}}}{(\sqrt{3}\ n)^{0.157 \log_3((n +
1.43)/24^{1/3})}},  \] and hence
\[2^{\frac{n}{3}}\ \ <\ n\ \cdot\ (\sqrt{3}\ n)^{0.157 \log_3((n + 1.43)/24^{1/3})}.\]
This turns out to be true only when $n\leq 15$, as claimed.\\

{\bf Step 6.} $n>k$, where we recall that $k=\log_2|W|$.\\

Suppose that $n=k$, so by Steps 4 and 5 we have $G\leq \Gamma(2^n)$
where $n\leq 15$. If $n=1$, we get a contradiction immediately, so
we may assume that $n\geq 3$ (as $|G|$ is odd).

Write $\Gamma_0(2^n)$ for the cyclic subgroup of $\Gamma(2^n)$
consisting of scalar multiplications when considered acting on the
field of $2^n$ elements. Let $K=G\cap \Gamma_0(2^n)$, so $K$ is
cyclic and normal in $G$ of order dividing $2^n-1$ and acts
frobeniusly on $V$. Let $l=|G/K|$, so $l$ divides $n$. If $l=n$,
then we get $k(GV)> k(GV/(KV))=k(G/K)=l=n$, a contradiction.
Therefore, since $|G|$ is odd and $n\geq 3$, we get $l\leq n/3$.
%Now observe that
%\[\ \ \ n\ \geq\ k(GV)\ >\ k(G)\geq\ k(G/K)+\frac{|K|-1}{|G/K|}\ =\
%l+\frac{|K|-1}{l}\]
%and hence $(n-l)l\geq |K|-1$. Now the function $f(l)=(n-l)l$
%is a parabola which is open downward and takes its
%maximum value at $l=n/2$, but since $1\leq l\leq n/3$, $f$ will take its maximum value at one of the endpoints. This shows that
%\[|K|-1\ \leq\ \mbox{max}\{f(1), f(n/3)\}\ =\  \mbox{max}\{n-1, 2n^2/9.\}\]
%For the moment, let $n\geq 3$, then that maximum is $2n^2/9$,
%so we get $|K|-1\leq 2n^2/9$. Therefore

Now we have
\[\ \frac{n^2}{3} \ \geq\ l\cdot n\ \geq l\ k(GV)\ \geq k(KV)\ \geq\ |K|+\frac{|V|-1}{|K|}.\]
It is an easy exercise to check that the function
$f(x)=x+\frac{|V|-1}{x}$  ($0<x<|V|$) attains its minimum value at
$x=\sqrt{|V|-1}$, and this minimum value is $2\sqrt{|V|-1}$. It
follows that
\[\frac{n^2}{3}\ \geq 2\sqrt{|V|-1}\  =\ \sqrt{4|V|-4}\ \geq\ \sqrt{3|V|}\ =\ \sqrt{3} \cdot 2^{n/2}.\]
(The last inequality holds since $n\geq 3$.) Thus $n^2\geq\
3\sqrt{3}\cdot 2^{n/2}$, and it is easy to verify that
this inequality holds for no positive integer $n$. This contradiction completes the proof of Step 6.\\

{\bf Step 7.} Final contradiction.\\

We keep using the above notation. Clearly, the action of $G$ on $V$
will have at least $m=n/k$ distinct orbits, and each of these orbits
is a conjugacy class of $GV$. Hence
\[(++)\ \ \ \ k(GV)\geq \frac{n}{k} +1,\]
 and we
will use this as we go through the remaining small cases. Recall
that $k(GV)\leq n$, we must have $k>1$.

By Steps 5 and 6 and the facts that $|G|$ is odd and $k>1$, the
remaining cases that we have to look at are the following:
\begin{itemize}
\item $n=15$, $k=3$;

\item $n=15$, $k=5$;

\item $n=14$, $k=2$;

\item $n=12$, $k=4$;

\item $n=12$, $k=2$;

\item $n=10$, $k=2$;

\item $n=9$, $k=3$;

\item $n=6$, $k=2$.
\end{itemize}

We now consider each of these cases separately.\\

$n=15$, $k=3$: Since the only odd order subgroups of $S_5$
containing a 5-cycle are the Sylow $5$-subgroups, we have $G
\lesssim \Gamma(2^3) \wr Z_5$ and $|G|$ divides $7^5\cdot3\cdot
5\cdot 5$. We get six conjugacy classes from $(++)$ and four more
from elements of order 5. If $3^3$ divides $|G|$, then, as $Z_5$
permutes the elements of order three in orbits of lengths 1 and
five, we get six more conjugacy classes of $G$ from elements of
order 3, contradicting $k(GV)\leq n=15$. So $|G|$ is not divisible
by $3^3$. In particular, the Hall $\{3,5\}$-subgroup of $G$ has
order no more than 45. If $7^3$ divides $|G|$, then we see that we
get more than $7^3/45$, that is, at least eight more classes of $GV$
from elements of order 7, again leading to a contradiction to
$k(GV)\leq 15$. Hence $|G|$ divides $5\cdot 3^2\cdot 7^2$. Now $G$
acts on the nonzero elements of $V$, and hence $k(GV)\geq
(|V|-1)/|G| +1\geq 8^5/(45\cdot 49) +1>15.5$, so
$k(GV)\geq 16$, again a contradiction, concluding this case.\\

$n=15$, $k=5$: Then $G \lesssim \Gamma(2^5) \wr Z_3$. From $(++)$ we
get four classes of $GV$, and from elements of order three at least
two more classes. If $5^3$ divides $G$, we get $(5^3-1)/3
>40$ more classes from elements of order 5, a contradiction. So
$|G|$ is not divisible by $5^3$. Also, if $31^2$ divides $|G|$, then
we get at least $(31^2-1)/(5^2\cdot3)>12$ conjugacy classes from
elements of order 31, bringing the total for $k(GV)$ above 15, a
contradiction. Thus we now know that $|G|$ divides $3\cdot 5^2\cdot
31$, and as in the previous case we then see that $k(GV)\geq
(|V|-1)/|G| +1\geq 8^5/(75\cdot 31)>15$, the final
contradiction in this case.\\

$n=14$, $k=2$: Then by \cite[Corollary 2.25]{manz-wolf} we know that
$G \lesssim \Gamma(2^2) \wr (Z_7 \cdot Z_3)$, so by the oddness of
$|G|$, we even have $G \lesssim Z_3 \wr (Z_7 \cdot Z_3)$.  From
$(++)$ we get eight classes of $GV$ inside $V$. Bringing back the
notation from Step 5, we recall that $V_N=W_1\oplus\dots\oplus W_7$
for subspaces $W_i$s which are transitively permuted by $G$, and we
write $K$ for the kernel of this permutation action. Then $K$ is an
elementary abelian 3-group of order at most $3^7$, and $|G/K|$
divides 21. Since clearly 7 divides $|G|$, from $G/K$, which is
cyclic of order 7 or Frobenius of order 21, we get five conjugacy
classes. Now if $3^4$ divides $|K|$, then we get an additional at
least $(3^4-1)/21$, that is three conjugacy classes of $GV$ coming
from (nontrivial) elements in $K$, bringing the total above 14, a
contradiction. Hence $|K|$ divides $3^3$ and hence $|G|$ divides
$27 \cdot 21$. But then as before we see that $k(GV)\geq
(|V|-1)/|G|\geq (2^{14}-1)/(27 \cdot 21)
>28$, which gives a final contradiction in this case.\\

$n=12$, $k=4$: Then $G \lesssim \Gamma(2^4) \wr Z_3$ and then, as
$|G|$ is odd, even $G \lesssim Z_{15}\wr Z_3$.  From $(++)$ we get
four classes, and from the ``$Z_3$ on top'' we get two additional
classes. If $5^2$ divides $|G|$, then from elements of order 5 we
get $(5^2-1)/3=8$ more classes, bringing the number of conjugacy
classes of $GV$ to 12. So there cannot be any other elements in the
group, forcing $|G|=5^2\cdot 3=75$. But then $k(GV)\geq (|V|-1)/|G|
=(2^{12}-1)/75=4095/75 > 12$, a contradiction.

Now we have $5^2$ does not divide $|G|$. For the number $n(G,V)$ of
orbits of $G$ on $V$ we have $n(G, V)\geq (|V|-1)/|G|\geq \lceil
4095/(5\cdot 3^3\cdot 3)\rceil=11$, which together with the  two
classes mentioned above from the factor group of order 3 increases
the number of classes of $GV$ to more than 13, a contradiction.\\

$n=12$, $k=2$: In this case $G/N$ would be a transitive permutation
group of degree 6 and as such not be primitive, it follows that
$W_i$'s ($i=1,\dots, 6$) would split into three blocks of
imprimitivity of size 2 or two blocks of size 3. Since $G/N$ is of
odd order, it could not act
transitively on the $W_i$, a contradiction.\\

$n=10$, $k=2$: Then $G \lesssim \Gamma(2^2) \wr Z_5$ and hence even
 $G \lesssim Z_3 \wr Z_5$. By $(++)$ we have six classes of
$GV$ in $V$, and from the factor group of order 5 we get an
additional four classes, so these must be all conjugacy classes of
$GV$, implying
$|G|=5$, but then $k(GV)\geq (|V|-1)/|G|\geq 1023/5\geq 11$, a contradiction.\\

$n=9$, $k=3$: Here $G \lesssim \Gamma(2^3) \wr Z_3$,  so by $(++)$
we get four classes of $k(GV)$. Again recall that $V_N=W_1\oplus
W_2\oplus W_3$ for subspaces $W_i$s which are transitively permuted
by $G$, and $K$ is the kernel of this permutation action. Then $K$
is an abelian $\{3, 7\}$-group of order at most $15^3$, and $G/K$ is
cyclic of order 3. So from $G/K$ we get another two classes. Now if
$3^2$ divides $|K|$, then from elements of order 3 in $K$ we get
additional more than $8/3$, that is, three classes, implying $k(GV)=
10>n$, a contradiction. %and there being no other elements in $G$, so %$|G|=3^2\cdot 3$.
%But then $10=k(GV)\geq (|V|-1)/|G|\geq 511/27>10$, a contradiction.

Therefore $3^2$ does not divide $|K|$. Now if $7^2$ divides $|G|$,
then observe that we get at least $48/9$, thus six classes from
elements of order 7, and together with the four classes from $(++)$
this again leads to a contradiction. Thus $7^2$ does not divide
$|G|$, so $|G|$ divides 21 and thus  $k(GV)\geq (|V|-1)/|G|\geq
511/21>10$, the
final contradiction in this case.\\

$n=6$, $k=2$: In this last case, $G \lesssim \Gamma(2^2) \wr Z_3$
and thus $G \lesssim Z_3 \wr Z_3$.  From $(++)$ we get four classes,
and from the factor group of order three we get two more. This
forces $k(GV)=6$ and moreover $|G|=3$. Now $k(GV)\geq
(|V|-1)/|G|\geq 61/3>7$. This final contradiction concludes the
proof of the theorem.
\end{proof}

\begin{proof}[Proof of Theorem \ref{theorem1}]
Theorem~\ref{theorem1} follows from Theorem~\ref{theorem3} and
\cite[Theorem 1]{Malle-Spath}, as discussed in the Introduction.
\end{proof}

%\section{Acknowledgement} %\label{sec:Acknowledgement}
%The third author is supported by %a grant from the Simons %Foundation (No 499532).

%%%%%%%%%%%%%%%%%%%%%%%%%%%%%%%%%%%%%%%%%%%%%%%%%%%%%%%%%%%%%%%%%%%%%%%%%%%%%%%

\end{document}